\documentclass[10pt,reqno]{amsart}
\usepackage{amsmath, amsopn, amsfonts, amsthm, amssymb, amscd, enumerate, multicol, scalefnt}

\usepackage{amsmath, amsfonts, amssymb, amscd, enumerate, scalefnt}

\usepackage{etex}
\usepackage{hyperref}
\usepackage{t1enc}
\usepackage{graphicx}
\usepackage[all]{xy}
\usepackage{color}
\usepackage{slashed}
\usepackage[mathscr]{euscript}

\theoremstyle{plain}
\newtheorem{theorem}{Theorem}[section]
\newtheorem{definition}[theorem]{Definition}
\newtheorem{lemma}[theorem]{Lemma}
\newtheorem{proposition}[theorem]{Proposition}
\newtheorem{corollary}[theorem]{Corollary}
\newtheorem{remark}[theorem]{Remark}
\newtheorem{ex}[theorem]{Example}
\newtheorem*{notation}{Notation}
\newtheorem{thmint}{Theorem}

\renewcommand{\b}{\begin{equation}}
\newcommand{\e}{\end{equation}}

\newcommand{\g}{\mathfrak{g}}

\newcommand{\n}{\mathfrak{n}}

\renewcommand{\l}{\mathfrak{l}}

\newcommand\C{{\mathbb C}}
\newcommand\R{{\mathbb R}}

\newcommand{\tr}{\operatorname{tr}}
\newcommand{\id}{\operatorname{Id}}
\newcommand{\ad}{\operatorname{ad}}

\newcommand{\pp}{\operatorname{P}}

\newcommand{\ip}{\langle\cdot,\cdot\rangle}

\title[The Hull-Strominger system and the Anomaly flow on a class of solvmanifolds]{The Hull-Strominger system and the Anomaly flow on a class of solvmanifolds}

\author{Mattia Pujia}
\address{Dipartimento di Matematica G. Peano, Universit\`a di Torino, Via Carlo Alberto 10, 10123 Torino, Italy}
\email{mattia.pujia@unito.it}

\date{\today}

\subjclass[2010]{}
\thanks{The author was partially supported by G.N.S.A.G.A. of I.N.d.A.M.} 

\begin{document} 

\begin{abstract} We study the Hull-Strominger system and the Anomaly flow on a special class of 2-step solvmanifolds, namely the class of almost-abelian Lie groups. In this setting, we characterize the existence of invariant solutions to the Hull-Strominger system with respect to the family of Gauduchon connections in the anomaly cancellation equation. Then, motivated by the results on the Anomaly flow, we investigate the flow of invariant metrics in our setting, proving that it always reduces to a flow of a special form. Finally, under an extra assumption on the initial metrics, we show that the flow is immortal and, when the slope parameter is zero, it always converges to a K\"ahler metric, in Cheeger-Gromov topology.
\end{abstract}

\maketitle

\section{Introduction}
In \cite{CHSW}, Candelas, Horowitz, Strominger and Witten proposed a compactification of the 10-dimensional heterotic string given by the product of a maximal symmetric space-time with a 6-dimensional compact K\"ahler Calabi-Yau manifold, the so-called internal space. Shortly after, Hull \cite{Hull86} and Strominger \cite{Strom} independently considered a generalization of such compactification by allowing the product to be warped. In turn, the symmetric condition on the space-time, together with the ``warp factor'', gives rise to a complex system of PDEs on the internal space known as the Hull-Strominger system. More concretely, let $(X,\omega)$ be a 3-dimensional Hermitian manifold admitting a nowhere vanishing holomorphic $(3,0)$-form $\Psi$. Let also $(E,h)$ be a holomorphic vector bundle over $X$. Then, the {\em Hull-Strominger system} consists of
\begin{equation}\label{HS}
\begin{aligned}
&\omega^2\wedge F=0\,,\quad F^{2,0}=F^{0,2}=0\,,\\
&i\partial\overline\partial \omega = \frac{\alpha'}{4}\left( {\tr}(\Omega^\tau\!\wedge \Omega^\tau) - {\tr}(F\wedge F)\right)\,,\\
&d(\Vert\Psi\Vert_{\omega}\,\omega^2)=0\,,
\end{aligned}
\end{equation}
where $\Omega^\tau$ and $F$ are respectively the curvature of the Gauduchon connections $\nabla^\tau$ on $(X,\omega)$ and $\widetilde\nabla$ on $(E,h)$. 

In the above system, the first two equations represent the {\em Hermitian-Yang-Mills} equation for the connection $\widetilde\nabla$. The third equation is the so-called {\em anomaly cancellation} term (or {\em Bianchi identity}), which follows by the Green-Schwarz cancellation mechanism in string theory, while $\alpha'\in\R$ is the {\em slope parameter}. Finally, we mention that the last equation, the {\em conformally balanced} one (or {\em dilatino}), was originally formulated as
$$
d^\ast \omega=i(\bar\partial-\partial)\ln\Vert\Psi\Vert_\omega\,,
$$
and the above expression is due to Li and Yau \cite{LY05}.\medskip

In the present paper we focus on the Hull-Strominger system on almost-abelian Lie groups, a special class of 2-step solvable Lie groups. Under the assumption for the bundle of being `flat', we characterize the existence of invariant solutions to the system depending on the connection $\nabla^\tau$.

Let $(G,J)$ be a $6$-dimensional almost-abelian Lie group equipped with a left-invariant complex structure. Let also $\Psi$ be a left-invariant form on $G$ and $(E,h)$ a flat bundle, i.e. $F\equiv0$. Then, if we assume $\alpha'\in\R\backslash\{0\}$, our first result states as follows

\begin{thmint}\label{HS_sol_flat} If the almost-abelian Lie group $G$ is non-K\"ahler, then the Hull-Strominger system admits a left-invariant solution if and only if $\nabla^\tau$ is neither the Chern nor the Lichnerowicz connection.
\end{thmint}

Worthy enough, in the K\"ahler setting the above claim further simplifies, since every K\"ahler metric solves the Hull-Strominger system independently of the choice of $\alpha'$ and $\nabla^\tau$ (Proposition \ref{HS-kah}). Furthermore, almost-abelian K\"ahler Lie groups turn out to admit solutions to the {\em heterotic equation of motion} (Theorem \ref{HSI_sol}). We stress that, this result follows by a characterization of the instanton condition in our setting (Proposition \ref{Inst}). Indeed, in \cite{Ivan10}, Ivanov proved that a solution to the Hull-Strominger system, satisfying the extra assumption for $\nabla^\tau$ of being an instanton, gives rise to a solution to the heterotic equation of motion. Here, by an {\em instanton} we mean a connection $\nabla^\tau$ satisfying the Hermitian-Yang-Mills equation.\medskip

Although the Hull-Strominger system has been extensively investigated both from physicists and mathematicians, many it is still unknown and only few explicit solutions have been found over the years. In particular, the first solutions on compact non-K\"ahler manifolds were obtained by Fu and Yau on toric bundles over K3 surfaces, under the assumption for $\nabla^\tau$ of being Chern \cite{FY07, FY08} (see also \cite{FGV}, for a generalization of Fu-Yau's results over K3 orbifolds). Since then, other solutions to the system have been obtained in different settings, and we refer the reader to \cite{AG12,AG12_2,Fei16,FHP19,FIUV14,UV14}. 

Let us mention that, the Hull-Strominger system originally required the connections $\nabla^\tau$ and $\widetilde\nabla$ to be Chern. Nonetheless, as shown in \cite{BG14}, such an assumption may lead to non-existence results and hence different Hermitian connections can be considered (see also \cite{Strom}). In this direction, some results have been obtained by different authors over the years and we refer the reader to \cite{Bd89,CCAL,CCALMZ,FY15,FIUV,Hull86_3,IvIv,OUV}.

\medskip
Due to the complexity of the equations involved in the Hull-Strominger system, different strategies to detect explicit solutions have been proposed over the years. Among these, a fruitful approach seems to be the one proposed by Phong, Picard and Zhang, who introduced a coupled flow of Hermitian metrics whose stationary points are solutions to the Hull-Strominger system, namely the Anomaly flow \cite{PPZ18}. Concretely, the {\em Anomaly flow} is the flow of Hermitian metrics $(\omega_t,h_t)$, with $\omega_t$ on $X$ and $h_t$ along the fibers of $E$, satisfying
\begin{equation}\label{AF-G}
\begin{aligned}
\partial_t(\Vert\Psi\Vert_{\omega_t}\,\omega_t^2)\,\, = &\,\,\,  {i\partial\overline\partial \omega_t - \frac{\alpha'}{4}\left( {\rm tr}(\Omega_t^\tau\wedge \Omega_t^\tau) - {\rm tr}(F_t\wedge F_t)\right)}\,, \\
h_t^{-1} \partial_t\, h_t  \,\, = & \,\,\, \frac{\omega_t^2\wedge F_t}{\omega_t^3}\,,
\end{aligned}
\end{equation}
where $\Omega^\tau_t$ and $F_t$ are the curvature of $(X,\omega_t,\nabla^\tau)$ and $(E,h_t,\widetilde\nabla)$. 

\medskip
Our next results concern the Anomaly flow on almost-abelian Lie groups. More precisely, we assume $(G,J,\Psi)$ as in Theorem \ref{HS_sol_flat} and $(E,h_t)$ constantly flat, i.e. $F_t\equiv 0$. In this setting, the Anomaly flow reduces to
\begin{equation}\label{AF}
\partial_t(\Vert\Psi\Vert_{\omega_t}\,\omega_t^2) = i\partial\overline\partial \omega_t - \frac{\alpha'}{4} \tr(\Omega_t^\tau\!\wedge \Omega_t^\tau)\,,
\end{equation}
and we have the following

\begin{thmint}\label{AF-bal} The Anomaly flow \eqref{AF} preserves the balanced condition along left-invariant solutions, for any Gauduchon connection. Moreover, if the initial datum is balanced, the flow of left-invariant metrics reduces to
$$
\partial_t(\Vert\Psi\Vert_{\omega_t}\,\omega_t^2) =  f(\omega_t) \, i\partial\overline\partial \omega_t\,,
$$
where $f(\omega_t)$ is a real-valued function depending on $\alpha'$ and $\nabla^\tau$ (see \eqref{f(omega)}).
\end{thmint}

Our last result is about the long-time existence and convergence of the Anomaly flow.

\begin{thmint} \label{AF-long} Let $\omega_0$ be a left-invariant balanced metric on $(G,J)$, satisfying $f(\omega_0)>0$. Then, the Anomaly flow \eqref{AF} of left-invariant metrics starting at $\omega_0$ is defined for every $t\in[0,+\infty)$. Furthermore, in the special case $\alpha'=0$ (that is, $f(\omega_t)\equiv1$), the solution to the flow converges to an almost-abelian K\"ahler Lie group $(\bar G, \bar J,\bar \omega)$ as $t\to\infty$, in Cheeger-Gromov topology.
\end{thmint}

By convergence in the Cheeger-Gromov topology we mean that, for any increasing sequence of times there exists a subsequence $(t_k)_{k\in \mathbb{N}}$ for which the corresponding Hermitian manifolds converge in the following sense: there exist biholomorphisms $\varphi_k : \Omega_k \subset \bar G \to \varphi_k(\Omega_k) \subset G $ taking the identity of $\bar G$ to the identity of $G$, such that the open sets $\Omega_k$ exhaust $\bar G$, and in addition $\varphi_k^* g_{t_k} \to \bar g$ as $k\to\infty$, in $C^\infty$ topology uniformly over compact subsets.

\medskip
Let us recall that, in \cite{PPZ18}, Phong, Picard and Zhang showed that the Anomaly flow preserves the conformally balanced condition, once that $\nabla^\tau$ and $\widetilde \nabla$ are both Chern, and, under an extra assumption on the initial metric $\omega_0$, it is well-posed. Furthermore, in \cite{PPZ18_3}, they used the flow to give an alternative proof of the outstanding results obtained by Fu and Yau on the Hull-Strominger system. Finally, a generalization of the Anomaly flow to higher dimensions has also been proposed in \cite{PPZ19}, and it turned out to be strictly related to the positive Hermitian curvature flow introduced by Ustinovskiy in \cite{U19} (see \cite{FP19_2}). 

To conclude, the Anomaly flow has already been investigated in different homogeneous settings. In this direction, Phong, Picard and Zhang studied in \cite{PPZ19} the Anomaly flow on complex Lie groups; while, the author and Ugarte investigated the flow on a class of 2-step nilmanifolds \cite{PU}. In the latter case, it was proved that the Anomaly flow preserves the balanced and the locally conformally K\"ahler conditions for any couple of Gauduchon connections $\nabla^\tau$ and $\widetilde \nabla$, and explicit solutions to the flow were found, both with `flat' and `non-flat' bundle. Finally, some results on the Anomaly flow over almost-abelian Lie groups were recently obtained by Fino and Paradiso in \cite{FP20}.

\smallskip
For other relevant results on the Anomaly flow, we refer the reader to \cite{BV20, FHP19, FP19,PPZ18_2,PPZ18_4}; while, for some results on the Hermitian curvature flow in the homogeneous setting we refer the reader to \cite{AL,EFV15,LPV,PP,PodPan19,P,P20,S20,U18}.

\medskip
The paper is organized as follows. Section \ref{prel} is devoted to basic computations on almost-abelian Lie groups. In particular, we characterize the balanced condition in terms of the structure equations of the group. Then, by means of our results, we explicitly compute $\tr(\Omega^\tau\!\wedge \Omega^\tau)$. In section \ref{inv_sol}, we show that $i\partial \bar \partial \omega$ and $\tr(\Omega^\tau\!\wedge \Omega^\tau)$ are proportional, which in turn imply Theorem \ref{HS_sol_flat}. Finally, in Section \ref{flow_sec} we focus on the Anomaly flow \eqref{AF} on almost-abelian Lie groups. Theorem \ref{AF-bal} will follow again by the above mentioned proportional result, whereas to prove Theorem \ref{AF-long} we use the bracket flow technique. The behaviour of the flow on an concrete solvmanifold is investigated.

\medskip 
\noindent {\bf Acknowledgments.} The author warmly thanks Luigi Vezzoni and Luis Ugarte for their interest in the paper and many useful discussions. He also thanks Anna Fino, Fabio Paradiso and Francesco Pediconi for useful suggestions, and the anonymous referee for him/her constructive comments, which helped to improve the paper.

\section{Preliminaries}\label{prel}

\subsection{Almost-abelian Lie groups}\hfill\medskip

This section is devoted to basic computations in our class of solvable Lie groups. In particular, we derive necessary and sufficient conditions such that an almost-abelian Lie group admits a holomorphically trivial canonical bundle and a balanced Hermitian structure. \medskip

Let $G$ be a 6-dimensional {\em almost-abelian Lie group}, that is, the Lie algebra $\g$ of $G$ has a codimension-one abelian ideal $\n$. Let also $(\omega,J)$ be a left-invariant Hermitian structure on $G$. Then, there always exists a left-invariant coframe $\{e^1,\ldots,e^{6}\}$ on $G$ such that
\begin{equation}\label{on-bas}
\n:={\rm span}_\R\{ e_1,\ldots,e_{5}\}\,,\qquad \omega= e^{16}+e^{25}+e^{34}\,,\qquad Je^i=e^{7-i}\,, 
\end{equation}
where $i=\{1,2,3\}$ and $\{e_1,\ldots,e_6\}$ denotes the frame dual to $\{e^1,\ldots,e^6\}$ (see e.g. \cite[$\S$6]{LauVal}). 

On the other hand, if we denote by $\n_1:={\rm span}_\R\{ e_2,\ldots,e_{5} \}$, then the complex structure $J$ on $G$ is integrable if and only if
\begin{equation}\label{ad_6}
\ad_{\,e_6}=\left(\begin{matrix} a& 0 & 0 \\ v& A & 0\\ 0& 0 &0 \end{matrix}\right) \qquad \text{and}\qquad [A,J|_{\n_1}]=0\,,
\end{equation}
for some $a\in\R$, $v=(v_i)\in \n_1$ and $A=(A_{i}^j)\in\g\l(\n_1)$. As a consequence of \eqref{ad_6}, the coframe $\{e^1,\ldots,e^6\}$ satisfies the following structure equations
\begin{equation}\label{str-eq}
\left\{ \begin{aligned}
de^1&= a\, e^{16}\,,\\
de^2&=v_2\, e^{16}+A_2^2\, e^{26}+A_2^3\, e^{36}+A_2^4\, e^{46}+A_2^5\, e^{56}\,,\\
de^3&=v_3\, e^{16}+A_3^2\, e^{26}+A_3^3\, e^{36}+A_3^4\, e^{46}+A_3^5\, e^{56}\,,\\
de^4&=v_4\, e^{16}-A_3^5\, e^{26}-A_3^4\, e^{36}+A_3^3\, e^{46}+A_3^2\, e^{56}\,,\\
de^5&=v_5\, e^{16}-A_2^5\, e^{26}-A_2^4\, e^{36}+A_2^3\, e^{46}+A_2^2\, e^{56}\,,\\
de^6&=0\,.
\end{aligned}\right.
\end{equation}
\smallskip

\begin{notation}\rm In the following, we will denote by $(G,\omega,J)$ a 6-dimensional almost-abelian Lie group endowed with a left-invariant Hermitian structure, and by $\{e^1,\ldots,e^{6}\}$ a left-invariant $\omega$-unitary coframe on $G$ satisfying \eqref{on-bas}.
\end{notation}
Our first result characterizes almost-abelian Lie groups admitting holomorphically trivial canonical bundles. 
\begin{proposition}\label{hol_bun} $(G,J)$ admits a left-invariant nowhere vanishing holomorphic $(3,0)$-form if and only if $\tr A=-2a$ and $\tr JA=0$.
\end{proposition}

\begin{proof} Let us assume
$$
\Psi= (e^1+ie^6) \wedge (e^2+ie^5) \wedge (e^3+ie^4)\,.
$$
Then, by means of a direct computation, one gets 
$$
d\Psi=\Big((a+A_2^2+A_3^3)-i \left(A_2^5+A_3^4\right)\Big) \left(e^{1236}+i\,e^{1246}-i\,e^{1356}+e^{1456}\right)\,,
$$
and hence the claim follows.
\end{proof}

The following proposition characterizes the balanced condition in our setting.

\begin{proposition}\label{bal_str} The Hermitian structure $(\omega,J)$ on $G$ is balanced if and only if $\tr A=0$ and $v=0$.
\end{proposition}

\begin{proof} By definition, a Hermitian structure $(\omega,J)$ is said to be balanced if
$$
d^\ast \omega(X)=-\sum_{i=1}^6 (\nabla^{LC}_{e_i} \omega)(e_i,X)=0\,, \qquad \text{for any}\,\, X\in\g\,,
$$
where $\nabla^{LC}$ is the Levi-Civita connection of $\omega$. Therefore, in view of the Koszul formula, a direct computation yields that
$$
d^\ast\omega(X)=\tr \ad_{JX}+\frac 12 \sum_{i=1}^6 \omega(\ad_{J e_i} e_i, JX)\,,
$$
and the claim follows by setting $d^\ast \omega( e_i)=0$ for any $i=\{1,\ldots,6\}$.
\end{proof}

By means of Proposition \ref{hol_bun} and Proposition \ref{bal_str}, if $G$ admits a left-invariant balanced structure $(\omega,J)$ with holomorphically trivial canonical bundle, then $G$ is {\em unimodular} and the Lie bracket preserves $\n_1$. In particular, the structure equations reduce to
\begin{equation}\label{str-eq-red}
\left\{ \begin{aligned}
de^1&= de^6=0\,,\\
de^2&=A_2^2\, e^{26}+A_2^3\, e^{36}+A_2^4\, e^{46}+A_2^5\, e^{56}\,,\\
de^3&=A_3^2\, e^{26}-A_2^2\, e^{36}-A_2^5\, e^{46}+A_3^5\, e^{56}\,,\\
de^4&=-A_3^5\, e^{26}+A_2^5\, e^{36}-A_2^2\, e^{46}+A_3^2\, e^{56}\,,\\
de^5&=-A_2^5\, e^{26}-A_2^4\, e^{36}+A_2^3\, e^{46}+A_2^2\, e^{56}\,.
\end{aligned}\right.
\end{equation}

\begin{remark}\rm We stress that, unimodular solvable Lie groups of dimension 6 admitting nowhere vanishing holomorphic $(3,0)$-forms have been classified in \cite{FOU}. Furthermore, the results of Proposition \ref{hol_bun} and Proposition \ref{bal_str} have also been obtained by Fino and Paradiso in \cite{FP20}, using different techniques.
\end{remark}

\subsection{Trace of the curvature}\hfill\medskip

\noindent  

In the following, we explicitly compute the trace of the curvature 4-form $\Omega^\tau\!\wedge \Omega^\tau$, for any Gauduchon connection $\nabla^{\tau}$ on our class of almost-abelian Hermitian Lie groups.  The convention we adopt is the same of \cite{PU} (see also \cite{FIUV,OUV}). \medskip

Let $M$ be a smooth manifold equipped with a Hermitian structure $(\omega,J)$. In \cite{Gaud}, Gauduchon introduced a 1-parameter family $\{\nabla^{\tau}\}_{\tau\in \mathbb{R}}$ of Hermitian connections on the tangent bundle $TM$ (that is, $\nabla^\tau J=0$ and $\nabla^\tau \omega=0$) distinguished by the properties of their torsion tensors, the so-called {\em canonical connections}. Concretely, any connection in the Gauduchon family satisfies
\begin{equation}\label{family_G}
\omega(\nabla^{\tau}_XY,JZ) = \omega(\nabla^{LC}_X Y,J Z)  + \frac{\tau-1}{4}\,d^c\omega(X,Y,Z) + \frac{\tau+1}{4}\,d^c\omega(X,JY,JZ)\,,
\end{equation}
where $\nabla^{LC}$ is the Levi-Civita connection of $(M,\omega)$, and $d^c=(-1)^rJdJ$ is the {\em real Dolbeault operator} acting on $r$-forms. 

We stress that, by setting $\tau=1$ one gets the Chern connection, while $\tau=-1$ recovers the Bismut connection. Finally, the value $\tau=0$ corresponds to the Lichnerowicz connection.

\medskip
Now, let $G$ be a 6-dimensional Lie group equipped with a left-invariant Hermitian structure $(\omega,J)$. Let also $\{e_1,\ldots,e_6\}$ be a left-invariant $\omega$-unitary frame of $G$ and $\{e^1,\ldots,e^6\}$ its dual. Then, the {\em connection $1$-forms} $\sigma^i_j$ associated to a linear connection $\nabla$ are given by
$$
\sigma^i_j(e_k) := \omega(J(\nabla_{e_k}e_j), e_i),
$$
or, equivalently, $\nabla_X e_j = \sigma^1_j(X)\, e_1 +\cdots+ \sigma^6_j(X)\, e_6$. Similarly, the {\em curvature $2$-forms} $\Omega^i_j$ of $\nabla$ can be defined via
\begin{equation}\label{curvature}
\Omega^i_j := d \sigma^i_j + \sum_{1\leq k \leq 6}\sigma^i_k\wedge \sigma^k_j\,,
\end{equation}
while the trace of the 4-form  $\Omega\wedge\Omega$ is given by
\begin{equation}\label{traceRmRm}
\tr(\Omega\wedge\Omega)= \sum_{1\leq i<j\leq 6} \Omega^i_j\wedge\Omega^i_j\,.
\end{equation}

On the other hand, given a canonical connection $\nabla^\tau$ on $(G,\omega,J)$, the connection 1-forms $(\sigma^\tau)^i_j$ can be explicitly obtained as follows. Let $c_{ij}^k$ be the structure constants of $G$ with respect to $\{e^1,\ldots,e^6\}$, i.e.
 $$
de^k = \sum_{1\leq i<j \leq 6} c_{ij}^k \, e^{i j},\quad\quad k=1,\ldots,6.
$$
Then, by means of the Koszul formula, the connection 1-forms $(\sigma^{LC})^i_j$ of the Levi-Civita connection satisfy
$$
(\sigma^{LC})^i_j(e_k) = -\frac12 \left( \omega(Je_i,[e_j,e_k]) - \omega(Je_k,[e_i,e_j]) +\omega(Je_j,[e_k,e_i]) \right)=\frac12(c^i_{jk}-c^k_{ij}+c^j_{ki})\,,
$$
and hence, in view of \eqref{family_G}, the connection 1-forms $(\sigma^{\tau})^i_j$ can be written as follows
\begin{equation}\label{1-form}\begin{aligned}
 (\sigma^{\tau})^i_j(e_k)
 =&\frac12(c^i_{jk}-c^k_{ij}+c^j_{ki}) +\frac{\tau-1}{4}\,d\omega(J e_k, J e_j, J e_i) + \frac{\tau+1}{4}\, d\omega(J e_k,e_j,e_i).
\end{aligned}\end{equation}

We are now in a position to compute the trace of $\Omega^{\tau}\!\wedge\Omega^{\tau}$ for any canonical connection $\{\nabla^{\tau}\}_{\tau\in \mathbb{R}}$ on our class of Lie groups.

\begin{proposition}\label{tr-bas} Let $(G,\omega,J)$ be a 6-dimensional almost-abelian Lie group equipped with a left-invariant balanced structure. Let also $(G,J)$ admit a nowhere vanishing holomorphic $(3,0)$-form. Then, for any Gauduchon connection $\nabla^\tau$ on $(G,\omega,J)$, the trace of the curvature $\Omega^\tau\!\wedge\Omega^\tau$ satisfies
\begin{equation*}\begin{aligned}
\tr(\Omega^\tau\!\wedge \Omega^\tau)=&\frac{\tau(\tau-1)^2}{4}\Big[4\left(A_2^2\right)^2+\left(A_2^3+A_3^2\right)^2+\left(A_2^4-A_3^5\right)^2\Big]\cdot\Big\{\\
&\qquad\qquad-\Big[A_2^2\left(A_2^4+A_3^5\right)-A_2^5 \left(A_2^3+A_3^2\right)\Big] \left(e^{1236}-e^{1456}\right)\\
&\qquad\qquad+\Big[A_2^2\left(A_2^3-A_3^2\right)+A_2^5 \left(A_2^4-A_3^5\right)\Big] \left(e^{1246}+e^{1356}\right)\\
&\qquad\qquad+\Big[2\left(A_2^2\right)^2+\left(A_3^2\right)^2+\left(A_3^5\right)^2+A_2^3A_3^2-A_2^4A_3^5\Big]\, e^{1256}\\
&\qquad\qquad+\Big[2\left(A_2^2\right)^2+\left(A_2^3\right)^2+\left(A_2^4\right)^2+A_2^3A_3^2-A_2^4A_3^5\Big]\, e^{1346}\Big\}\,,
\end{aligned}\end{equation*}
where $\{e^1,\ldots,e^6\}$ is a left-invariant coframe of $G$ satisfying \eqref{on-bas}
\end{proposition}

\begin{proof} Let $(\sigma^\tau)_i^j$ be the connection 1-forms of the canonical connection $\nabla^\tau$ on $(G,\omega,J)$. Since $\nabla^\tau$ is a Hermitian connection, we have $(\sigma^\tau)_i^j=-(\sigma^\tau)_j^i$. On the other hand, a direct computation yields
$$
\begin{aligned} 
(\sigma^\tau)_2^1=&\, \frac{(\tau-1)}{4}(A_3^5-A_2^4)\,e^3+\frac{(\tau-1)}{4}(A_2^3+A_3^2)\,e^4+\frac{(\tau-1)}{2}A_2^2 \,e^5\,,\\
(\sigma^\tau)_3^1=&\, \frac{(\tau-1)}{4}(A_2^4-A_3^5)\,e^2-\frac{(\tau-1)}{2}A_2^2 \,e^4+\frac{(\tau-1)}{4}(A_2^3+A_3^2)\,e^5\,,\\
(\sigma^\tau)_3^2=&\, \frac{\tau}{2}(A_3^5-A_2^4)\,e^1+\frac12 (A_3^2-A_2^3)\,e^6\,,\\
(\sigma^\tau)_4^1=&- \frac{(\tau-1)}{4}(A_2^3+A_3^2)\,e^2+\frac{(\tau-1)}{2}A_2^2 \,e^3+\frac{(\tau-1)}{4}(A_2^4-A_3^5)\,e^5\,,\\
(\sigma^\tau)_4^2=&\, \frac{\tau}{2}(A_2^3+A_3^2)\,e^1-\frac12 (A_2^4+A_3^5)\,e^6\,,\\
(\sigma^\tau)_4^3=&- \tau A_2^2\,e^1+A_2^5\,e^6\,,\\
(\sigma^\tau)_5^1=&- \frac{(\tau-1)}{2}A_2^2\,e^2- \frac{(\tau-1)}{4}(A_2^3+A_3^2)\,e^3+\frac{(\tau-1)}{4}(A_3^5-A_2^4)\,e^4\,,\\
(\sigma^\tau)_6^1=&\,0\,,
\end{aligned}
$$
together with the following relations:
$$\begin{aligned}
&(\sigma^\tau)_6^5=-(\sigma^\tau)^1_2\,,\quad (\sigma^\tau)_6^4=-(\sigma^\tau)_3^1\,,\quad (\sigma^\tau)_6^3=(\sigma^\tau)_4^1\,,\quad (\sigma^\tau)_6^2=(\sigma^\tau)_5^1\,,\\
&(\sigma^\tau)_5^4=-(\sigma^\tau)_3^2\,,\quad (\sigma^\tau)_5^3=(\sigma^\tau)_4^2\,,\quad (\sigma^\tau)_5^2=-(\sigma^\tau)_4^3\,.
\end{aligned}$$ 
Then, by means of \eqref{curvature} and \eqref{traceRmRm}, it follows that
$$\begin{aligned}
\tr(\Omega^\tau\wedge\Omega^\tau)=& \sum_{i<j}\{(d(\sigma^\tau)^i_j + \sum_{k}(\sigma^\tau)^i_k\wedge (\sigma^\tau)^k_j)\wedge (d (\sigma^\tau)^i_j + \sum_{l}(\sigma^\tau)^i_l\wedge (\sigma^\tau)^l_j)\}\\
=& \sum_{i<j}\{d (\sigma^\tau)^i_j\wedge d(\sigma^\tau)^i_j + 2\, d (\sigma^\tau)^i_j\wedge(\sum_{k}(\sigma^\tau)^i_k\wedge (\sigma^\tau)^k_j)+\sum_{k,l}(\sigma^\tau)^i_k\wedge (\sigma^\tau)^k_j\wedge(\sigma^\tau)^i_l\wedge (\sigma^\tau)^l_j\}\\
=& \sum_{i<j}2\,  d (\sigma^\tau)^i_j\wedge(\sum_{k}(\sigma^\tau)^i_k\wedge (\sigma^\tau)^k_j)\,.
\end{aligned}$$
Indeed, since $G$ is almost-abelian, one directly gets $d (\sigma^\tau)^i_j\wedge d(\sigma^\tau)^i_j=0$, while a manipulation of the indexes yields to
$$\begin{aligned}
&\sum_{i<j}\sum_{k,l}(\sigma^\tau)^i_k\wedge (\sigma^\tau)^k_j\wedge(\sigma^\tau)^i_l\wedge (\sigma^\tau)^l_j\\
=&\sum_{i<j}\sum_{k<l}(\sigma^\tau)^i_k\wedge (\sigma^\tau)^k_j\wedge(\sigma^\tau)^i_l\wedge (\sigma^\tau)^l_j+\sum_{i<j}\sum_{l<k}(\sigma^\tau)^i_k\wedge (\sigma^\tau)^k_j\wedge(\sigma^\tau)^i_l\wedge (\sigma^\tau)^l_j\\
=&\sum_{i<j}\sum_{k<l}(\sigma^\tau)^i_k\wedge (\sigma^\tau)^k_j\wedge(\sigma^\tau)^i_l\wedge (\sigma^\tau)^l_j-\sum_{i<j}\sum_{k<l}(\sigma^\tau)^k_j\wedge (\sigma^\tau)^j_l\wedge(\sigma^\tau)^k_i\wedge (\sigma^\tau)^i_l\\
=&\sum_{i<j}\sum_{k<l}(\sigma^\tau)^i_k\wedge (\sigma^\tau)^k_j\wedge(\sigma^\tau)^i_l\wedge (\sigma^\tau)^l_j-\sum_{i<j}\sum_{k<l}(\sigma^\tau)^i_l\wedge (\sigma^\tau)^l_j\wedge(\sigma^\tau)^i_k\wedge (\sigma^\tau)^k_j=0\,.
\end{aligned}$$
Finally, in view of Appendix A, we have
$$
\tr(\Omega^\tau\wedge\Omega^\tau)\in{\rm span}\{e^{1rs6}\}\,,
$$
with $r,s\in\{2,3,4,5\}$, and the claim follows by a straightforward computation.
\end{proof}

By the proof of the previous proposition, we have the following two corollaries.

\begin{corollary} Under the hypothesis of Proposition \ref{tr-bas}. Given a left-invariant nowhere vanishing holomorphic $(3,0)$-form $\Psi$ on $(G,J)$, it follows $\nabla^\tau \Psi= 0$ for any $\tau\in\R$. In particular, any Gauduchon connection $\nabla^\tau$ is a $SU(3)$-connection.
\end{corollary}

\begin{proof}
Let $\Psi$ be the (3,0)-form given by $\Psi=(e^1+ie^6)\wedge(e^2+ie^5)\wedge(e^3+ie^4)$. Since
$$\nabla^\tau\Psi=0 \quad \text{if and only if}\quad (\sigma^\tau)_6^1+(\sigma^\tau)_5^2+(\sigma^\tau)_4^3=0\,,$$
the claim follows by a direct computation.
\end{proof}

\begin{corollary}\label{kahler} Under the hypothesis of Proposition \ref{tr-bas}. If $(\omega,J)$ is a K\"ahler structure on $G$, then $\tr(\Omega^\tau\!\wedge \Omega^\tau)$ identically vanishes.
\end{corollary}

\begin{proof} Since the K\"ahler condition corresponds to $A\in\mathfrak u(\n_1)$, the claim directly follows by \eqref{str-eq-red}.
\end{proof}

\section{Invariant solutions to the Hull-Strominger system}\label{inv_sol}
In this section, we investigate the existence of left-invariant solutions to the Hull-Strominger system on almost-abelian Lie groups. Necessary and sufficient conditions to the existence of such solutions will be obtained, under the assumption for $\pi:E\to G$ of being a flat holomorphic bundle. In particular, we will prove Theorem \ref{HS_sol_flat}.

\medskip
From now on, let $G$ be a 6-dimensional almost-abelian Lie group equipped with a left-invariant Hermitian structure $(\omega,J)$. Let also $\{e^1,\ldots,e^6\}$ be a left-invariant $\omega$-unitary coframe satisfying \eqref{on-bas} and $\Psi$ a left-invariant nowhere vanishing holomorphic $(3,0)$-form on $(G,J)$, that is,
\begin{equation}\label{hol-form}
\Psi=k\, (e^1+ie^6)\wedge(e^2+ie^5)\wedge(e^3+ie^4)\,,
\end{equation}
for some $k\in\R\backslash\{0\}$. 

\medskip
Our next result will be fundamental to prove both Theorem \ref{HS_sol_flat} and Theorem \ref{AF-bal}.
\begin{lemma}\label{ddbar} Let $(\omega,J)$ be a balanced structure on $G$. Then, for any Gauduchon connection $\nabla^\tau$ on $(G,\omega,J)$, we have
$$
\tr(\Omega^\tau\!\wedge\Omega^\tau) = K(\omega,\tau)\, i\partial\bar\partial \omega\,,
$$
with $K(\omega,\tau)\in\R$ satisfying
$$K(\omega,\tau)= \frac{\tau(\tau-1)^2}{8}\tr_\omega(S(A)^2)\,.$$
Here, $S(A)=\tfrac12(A+A^*)$ denotes the symmetric part of $A$.
\end{lemma}

In the following, $(\cdot)^*$ will denote the transpose with respect to $\omega$.

\begin{proof}[Proof of Lemma \ref{ddbar}]Since $d^c=\frac i2(\bar\partial-\partial)$, a direct computation yields
$$\begin{aligned}
i\partial \bar\partial \omega=dd^c \omega=& -2\Big[A_2^2\left(A_2^4+A_3^5\right)-A_2^5 \left(A_2^3+A_3^2\right)\Big] \left(e^{1236}-e^{1456}\right)\\
&+2\Big[A_2^2\left(A_2^3-A_3^2\right)+A_2^5 \left(A_2^4-A_3^5\right)\Big]\left(e^{1246}+e^{1356}\right)\\
&+2\Big[2\left(A_2^2\right)^2+\left(A_3^2\right)^2+\left(A_3^5\right)^2+A_2^3A_3^2-A_2^4A_3^5\Big] e^{1256}\\
&+2\Big[2\left(A_2^2\right)^2+\left(A_2^3\right)^2+\left(A_2^4\right)^2+A_2^3A_3^2-A_2^4A_3^5\Big] e^{1346}\,.
\end{aligned}$$
Hence, in view of Proposition \ref{tr-bas}, it is enough to set
$$
K(\omega,\tau)= \frac{\tau(\tau-1)^2}{8}\Big[4\left(A_2^2\right)^2+\left(A_2^3+A_3^2\right)^2+\left(A_2^4-A_3^5\right)^2\Big]\,,
$$
and the claim follows.
\end{proof}

As a direct consequence of the above lemma, the trace $\tr(\Omega^\tau\!\wedge\Omega^\tau)$ turns out to be a (2,2)-form for any Gauduchon connection $\nabla^\tau$ on $(G,\omega,J)$. We stress that, an analogue result was recently obtained by the author and Ugarte for a class of 2-step nilpotent nilmanifolds (see \cite{PU}); while, it is a well-known fact for complex unimodular Lie groups (see \cite{FY15}).

\smallskip
We are now in a position to prove Theorem \ref{HS_sol_flat}.

\begin{proof}[Proof of Theorem \ref{HS_sol_flat}] Let $(\omega,J)$ be a balanced non-K\"ahler structure on $G$ and $\pi:E\to G$ a flat holomorphic bundle. Let also $\nabla^\tau$ be a Gauduchon connection on $(G,\omega,J)$. By means of Lemma \ref{ddbar}, if $\tau\neq\{0,1\}$ (that is, $\nabla^\tau$ is neither the Chern nor the Lichnerowicz connection), there always exists $\alpha'\in\R\backslash\{0\}$ such that $(\omega,J)$ solves the Hull-Strominger system \eqref{HS}. On the other hand, if $(\omega,J)$ solves the Hull-Strominger system \eqref{HS} for some $\alpha'\in\R\backslash\{0\}$, then $\tau\neq\{0,1\}$, and hence the claim follows.
\end{proof}

It is worth noting that, Theorem \ref{HS_sol_flat} completely characterizes the existence of solutions to the Hull-Strominger system \eqref{HS} in the non-K\"ahler setting. On the other hand, in the K\"ahler case we have the following proposition, whose proof directly follows by Corollary \ref{kahler}.

\begin{proposition}\label{HS-kah} Let $(\omega,J)$ be a K\"ahler structure on $G$ and $\pi:E\to G$ a flat holomorphic bundle. Then, $(\omega,J)$ solves the Hull-Strominger system \eqref{HS} for any Gauduchon connection $\nabla^\tau$ and $\alpha'\in\R$.
\end{proposition}

We conclude the section by focusing on Gauduchon connections satisfying the instanton condition. Let us recall that, a connection $\nabla$ is said to be an {\em instanton} if its curvature $\Omega^\nabla$ satisfies the Hermitian-Yang-Mills equation
\begin{equation}\label{instanton_cond}
\omega^2\wedge \Omega^\nabla=0\,,\quad (\Omega^\nabla)^{2,0}=(\Omega^\nabla)^{0,2}=0\,.
\end{equation}

\begin{proposition}\label{Inst} Let $(\omega,J)$ be a balanced structure on $G$. If the Gauduchon connection $\nabla^\tau$ on $(G,\omega,J)$ is an instanton with respect to $\omega$, then it has to be a flat. Moreover, for any $\tau\neq1$, $\nabla^\tau$ is an instanton if and only if the metric $\omega$ is K\"ahler.
\end{proposition}

\begin{proof} Let $\{e_1,\ldots,e_6\}$ be a left-invariant frame of $G$ satisfying \eqref{on-bas}. In terms of such a frame, the instanton condition \eqref{instanton_cond} reduces to
$$\begin{aligned}
&(\Omega^\tau)^i_j(e_1,e_6)+(\Omega^\tau)^i_j(e_2,e_5)+(\Omega^\tau)^i_j(e_3,e_4)=0\,,\\
&(\Omega^\tau)^i_j(e_k,e_l)-(\Omega^\tau)^i_j(Je_k,Je_l)=0\,,
\end{aligned}$$
for any $1\leq i,j,k,l\leq6$. 

Now, let $\nabla^\tau$ be an instanton with respect to $\omega$. By means of the curvature 2-forms $(\Omega^\tau)_i^j$ given in the Appendix B, it follows
$$
(\Omega^\tau)^1_6(e_1,e_6)+(\Omega^\tau)^1_6(e_2,e_5)+(\Omega^\tau)^1_6(e_3,e_4)=0
$$
if and only if
$$
\tau=1 \qquad \text{or}\qquad A_2^2=A_2^4-A_3^5=A_2^3+A_3^2=0\,.
$$
In the former case, i.e. $\tau=1$, one gets that $(\sigma^\tau)^i_j\in{\rm span}\{e^1,e^6\}$, for any $i,j\in\{1,\ldots,6\}$, which implies $(\Omega^\tau)^i_j\in {\rm span}\{e^1\wedge e^6\}$. Thus, the instanton condition reduces to
$$
(\Omega^\tau)^i_j(e_1,e_6)=0\,,
$$
and hence the connection has to be flat.

To conclude, let us assume $\tau\neq1$ and $A_2^2=A_2^4-A_3^5=A_2^3+A_3^2=0$ (that is, $A\in\mathfrak u(\n_1)$). In this setting, we have $(\sigma^\tau)^i_j\in{\rm span}\{e^6\}$, which yields $(\Omega^\tau)^i_j=0$, for any $i,j\in\{1,\ldots,6\}$. Therefore, since $A\in\mathfrak u(\n_1)$ is equivalent to require the metric $\omega$ to be K\"ahler, the claim follows.
\end{proof}

As a relevant application of the above proposition, we characterize almost-abelian Lie groups admitting solutions to the Hull-Strominger system with non-trivial instanton $\nabla^\tau$.

\begin{theorem}\label{HSI_sol} Let $(\omega,J)$ be a left-invariant Hermitian structure on $G$ solving to the Hull-Strominger system with flat bundle. If the Gauduchon connection $\nabla^\tau$ is an instanton with respect to $\omega$, then the Hermitian structure is K\"ahler.
\end{theorem}

\begin{proof} The claim directly follows by means of Proposition \ref{Inst}. Indeed, the instanton condition forces $\tr(\Omega^\tau\!\wedge\Omega^\tau)=0$, and hence the metric has to be K\"ahler, since it is both pluriclosed (that is, $\partial \bar \partial \omega=0$) and balanced.
\end{proof}

\section{The Anomaly flow on almost-abelian Lie groups}\label{flow_sec}

We now investigate the behaviour of the Anomaly flow on almost-abelian Lie groups. In particular, we show that the Anomaly flow always preserves the balanced condition and it reduces to a flow of a special form (Theorem \ref{AF-bal}). Finally, we use the bracket flow technique to prove the long-time existence of the flow (Theorem \ref{AF-long}).
\medskip

Let $(G,\omega_0,J)$ be a 6-dimensional almost-abelian Lie group equipped with a left-invariant balanced structure. Let also $\Psi$ be a left-invariant nowhere vanishing holomorphic $(3,0)$-form on $G$ (see \eqref{hol-form}). In the following, we will denote by $\{\omega_t\}$ the left-invariant solution to the Anomaly flow
\begin{equation}\label{AF-left}
\partial_t(\Vert\Psi\Vert_{\omega_t}\,\omega_t^2) =   i\partial\overline\partial\, \omega_t - \frac{\alpha'}{4} \tr(\Omega_t^\tau\!\wedge \Omega_t^\tau)\,,
\end{equation}
starting at $\omega_0$, and by $f(\omega_t)$ the real-valued function satisfying
\begin{equation}\label{f(omega)}
f(\omega_t)=1-\frac{\alpha'}{4}K(\omega_t,\tau)\,.
\end{equation}

We are now in a position to prove Theorem \ref{AF-bal}. 

\begin{proof}[Proof of Theorem \ref{AF-bal}] Let $\{\widetilde \omega_t\}$ be the family of left-invariant  Hermitian metrics solving
\begin{equation*}
\partial_t(\Vert\Psi\Vert_{\widetilde \omega_t}\,\widetilde\omega_t^2) = f(\widetilde\omega_t)\, i\partial\overline\partial\, \widetilde\omega_t\,,\qquad {\widetilde \omega_t}{|_0}=\omega_0\,.
\end{equation*} 
Then, it is not hard to prove that the solution $\{\widetilde \omega_t\}$ holds balanced along the flow. Indeed, since $f(\widetilde \omega)$ is left-invariant by construction, one gets 
$$
\partial_t(d(\Vert\Psi\Vert_{\widetilde \omega_t}\,\widetilde\omega_t^2))=d\,\partial_t(\Vert\Psi\Vert_{\widetilde \omega_t}\,\widetilde\omega_t^2) = d(f(\widetilde\omega_t)\, i\partial\overline\partial\, \widetilde\omega_t)=0\,,
$$
which implies $d(\Vert\Psi\Vert_{\widetilde \omega_t}\,\widetilde\omega_t^2)=0$ for every $t$, if $d(\Vert\Psi\Vert_{\widetilde \omega_0}\,\widetilde\omega_0^2)=0$, and hence the solution holds balanced.

On the other hand, by means of Lemma \ref{ddbar}, the family of left-invariant balanced metrics $\{\widetilde\omega_t\}$ also solves the Anomaly flow \eqref{AF-left}, and hence the claim follows by uniqueness of ODEs solutions.
\end{proof}

Let us now focus on the {\em reduced} Anomaly flow
\begin{equation}\label{AF-red}
\partial_t(\Vert\Psi\Vert_{\omega_t}\,\omega_t^2) = f(\omega_t)\, i\partial\overline\partial\, \omega_t\,, 
\end{equation}
starting at $\omega_0$. 
If we consider the rescaled metric $\nu_t=\Vert\Psi\Vert_{\omega_t}^{\frac 12}\,\omega_t$, then the reduced flow can be reformulated as follows
\begin{equation*}
\partial_t\,\nu_t^2 = \Vert\Psi\Vert_{\nu_t}^{-2}\,f(\Vert\Psi\Vert_{\nu_t}^{-2}\nu_t)\,  i\partial\overline\partial\, \nu_t\,,
\end{equation*}
or, equivalently, 
\begin{equation}\label{nu_t}
\partial_t\,\nu_t = \Vert\Psi\Vert_{\nu_t}^{-2}\,f(\Vert\Psi\Vert_{\nu_t}^{-2}\nu_t)\, \iota_{\nu_t}( i\partial\overline\partial\, \nu_t)\,,
\end{equation}
where $\iota_{\nu}$ denotes the contraction with respect to $\nu$. 

In the following, we will use the bracket flow techniques to prove the long-time existence of the reduced Anomaly flow.

\subsection*{\bf The bracket flow technique}\label{bf_sub}\hfill\medskip

The bracket flow technique provides a method to translate prescribed geometric flows into flows of Lie brackets. Introduced by Lauret to study the Ricci flow on nilmanifolds (\cite{LauMatAnn}), it has extensively been used to investigate different geometric flows in Hermitian geometry (see e.g. \cite{AL}, \cite{LPV}, \cite{LauHer}, \cite{PV}).\medskip

Let $(\g,\mu_0)$ be the Lie algebra of $G$. By definition, the Lie bracket $\mu_0$ belongs to the so-called {\em variety of Lie algebras}
\[
\mu_0\in \mathcal C:=\left\{\mu\in\Lambda^2\g^*\otimes \g :\text{$\mu$ satisfies the Jacobi identity and $J$ is integrable}\right\}\,,
\]
which admits the `natural' action of 
$$
{\rm Gl}(\g,J) = \left\lbrace \varphi\in{\rm Gl}(\g) : [\varphi, J]=0\right\rbrace\simeq {\rm Gl}_3(\C)
$$ 
defined by
\begin{equation}\label{act_alg}
\varphi\cdot\mu:=\varphi\,\mu(\varphi^{-1}\cdot,\varphi^{-1}\cdot)\,.
\end{equation}

Now, let us denote by $\{\nu_t\}$ the family of left-invariant balanced metrics satisfying \eqref{nu_t}. Then, by means of \cite[Thm. 1.1]{LauHer}, there always exists a smooth curve $\{\varphi_t\}_{t\in I}\in{\rm Gl}(\g,J)$ such that 
\begin{equation}\label{autom}
\nu_t(\cdot,\cdot)=\nu_0(\varphi_t\cdot,\varphi_t\cdot)\,,
\end{equation}
and, at the same time, the family of Lie brackets
$$
\mu_t:= \varphi_t\cdot\mu_0
$$
satisfies the {\em bracket flow equation}
\begin{equation}\label{br_flow}
\frac{d}{dt}\mu_t = \pi(\pp_{\mu_t})\mu_t\,,\qquad {\mu_t}_{\vert_0}=\mu_0\,.
\end{equation}
Here, $\pi: {\rm End}(\g)\to {\rm End}(\Lambda^2\g^*\otimes \g)$ denotes the representation of the action defined in \eqref{act_alg}, that is,
\begin{equation}\label{act_brac}
\pi(E)\mu (\cdot,\cdot)= E\mu(\cdot,\cdot)-\mu(E\cdot,\cdot)-\mu(\cdot,E\cdot)\,,
\end{equation}
while $\pp_{\mu_t}:\g\to\g$ is the endomorphism satisfying
\begin{equation}\label{end_tens}
\pp_{\mu_t} = \varphi_t \pp_{\nu_t} \varphi_t^{-1}\,,\qquad \nu_t(\pp_{\nu_t}\cdot,\cdot)=P(\nu_t)\,,
\end{equation}
with $P(\nu_t)$ right-hand side of \eqref{nu_t}, i.e. $$P(\nu_t)=\Vert\Psi\Vert_{\nu_t}^{-2}\,f(\Vert\Psi\Vert_{\nu_t}^{-2}\nu_t)\, \iota_{\nu_t}( i\partial\overline\partial\, \nu_t)\,.$$

\subsection*{\bf Long-time existence and convergence of the flow}\hfill\medskip

Let us denote by $\ip$ the scalar product induced by $\nu_0$ on $(\g,J)$, which naturally gives rise to scalar products on any tensor product of $\g$ and its duals. Let also $\{e_1,\ldots,e_6\}$ be a $\ip$-unitary basis of $\g$ and $\mu=\mu(a,v,A)$ an almost-abelian Lie bracket on $(\g,J)$ satisfying \eqref{ad_6}. Then, in view of Section \ref{prel}, the balanced condition reduces to an algebraic condition on the bracket, which suggests our next definition.

\begin{definition}\label{bal_brac}\rm An almost-abelian Lie bracket $\mu=\mu(a,v,A)$ on $(\g,\ip,J)$ is said to be {\em balanced} if it satisfies $\tr A=0$ and $v=0$. 
\end{definition}

\begin{lemma}\label{oper} Let $\mu=\mu(0,0,A)$ be an almost-abelian balanced bracket on $(\g,\ip,J)$. Then, the endomorphism $\pp_\mu$ is given by
$$
\pp_\mu= \Vert\Psi\Vert_{\ip}^{-2} \,f(A) \left(\begin{matrix}\lVert A^+\rVert^2&0&0\\0&2[A^+,A^-]&0\\0&0&\lVert A^+\rVert^2\end{matrix}\right)\,.
$$
Here, the blocks are in terms of $\R e_1\oplus \n_1 \oplus \R e_6$, $f:{\g\l}(\n_1)\to \R$ is the map satisfying
$$
f(A)=1-\frac{\alpha'}{4}\frac{\tau(\tau-1)^2}{8}\Vert\Psi\Vert_{\ip}^{-2}\lVert A^+\rVert^2\,,
$$
and $A^\pm= (A\pm A^*)/2$.
\end{lemma}

\begin{proof} The claim follows by a straightforward computation, see also \cite[Lemma 6.3]{FP20}.
\end{proof}

Our next result will be fundamental to prove the long-time existence of the flow.

\begin{proposition}\label{ev_A}
Let $\mu_0=\mu_0(0,0,A_0)$ be an almost-abelian balanced Lie bracket on $(\g,J)$. Then, the bracket flow \eqref{br_flow} starting at $\mu_0$ reduces to
\begin{equation}\label{ev_A_t}
\tfrac{d}{dt}A = \Vert\Psi\Vert_{\ip}^{-2} \,f(A)\left( 2[[A^+,A^-],A]- \lVert A^+\rVert^2A\right)\,.
\end{equation}
Moreover, the flow consists entirely of almost-abelian balanced brackets $\mu_t=\mu_t(0,0,A_t)$, and the (3,0)-form $\Psi$ holds closed.
\end{proposition}

\begin{proof}
By means of Lemma \ref{oper}, the Anomaly flow preserves the splitting ${\R e_1\oplus \n_1 \oplus \R e_6}$, which in turn implies that the solution to the flow is of the form $\mu_t=\mu_t(a_t,v_t,A_t)$. Therefore, similarly to \cite[Prop. 4.11]{AL}, we have
$$
\frac{d}{dt}\left(\begin{matrix} a & 0\\ v & A
\end{matrix}\right) = \Vert\Psi\Vert_{\ip}^{-2} \,f(A)\left(\left[\left(\begin{matrix}\lVert A^+\rVert^2&0\\0&2[A^+,A^-]\end{matrix}\right),\left(\begin{matrix}a & 0 \\ v & A\end{matrix}\right)\right]- \lVert A^+\rVert^2\left(\begin{matrix}a & 0 \\ v & A\end{matrix}\right)\right)\,,
$$
or, equivalently, 
$$
\left\{
\begin{aligned}
\tfrac{d}{dt}{a} &= - \Vert\Psi\Vert_{\ip}^{-2} \,f(A)\,\lVert A^+\rVert^2\,a\,,\\
\tfrac{d}{dt}{v} &= \Vert\Psi\Vert_{\ip}^{-2} \,f(A) \left(2[A^+,A^-]-\lVert A^+\rVert^2\id\right)v\,,\\
\tfrac{d}{dt}{A} &= \Vert\Psi\Vert_{\ip}^{-2} \,f(A) \left(2[[A^+,A^-],A]- \lVert A^+\rVert^2A\right)\,.
\end{aligned}
\right.
$$
On the other hand, by uniqueness of ODE solutions we have $a_t=0$ and $v_t=0$, and hence the bracket flow reduces to the evolution equation
$$
\tfrac{d}{dt}A = \Vert\Psi\Vert_{\ip}^{-2} \,f(A)\left( 2[[A^+,A^-],A]- \lVert A^+\rVert^2A\right)\,.
$$

Finally, the second statement directly follows by \eqref{ev_A_t}. Indeed, since $J$ is fixed we have $[A_t,J|_{\n_1}]=0$, and a direct computation yields
$$
\tfrac{d}{dt} \tr (A_t) = -\Vert\Psi\Vert_{\ip}^{-2} \,f(A_t)\, \tr(A_t) \qquad \text{and}\qquad \tfrac{d}{dt} \tr (JA_t)=-\Vert\Psi\Vert_{\ip}^{-2} \,f(A_t)\, \tr(JA_t)\,.
$$
Thus, we have $\tr(A_t)=\tr(JA_t)=0$ and, in view of Definition \ref{bal_brac} and Proposition \ref{hol_bun}, the claim follows.
\end{proof}

We are now in a position to prove Theorem \ref{AF-long}.

\begin{proof}[Proof of Theorem \ref{AF-long}] 
By means of Proposition \ref{ev_A}, to prove the long-time existence of the Anomaly flow, it is enough to show that $A_t$ is defined for every $t\in[0,+\infty)$. In this setting, let us assume $\Vert\Psi\Vert_{\ip}^{-2}=1$. By looking at how the norm of $A_t$ evolves, one gets
\begin{equation}\label{ev_norm}
\begin{aligned}
\tfrac{d}{dt}\lVert A\rVert^2 &= 2\langle \tfrac{d}{dt}A,A \rangle \\
&= 2\,f(A)\left(2\langle [[A^+,A^-],A],A\rangle-\langle \lVert A^+\rVert^2A,A\rangle\right)\\
&=2 \,f(A)\left(2\langle [[A^+, A^-],A],A\rangle- \lVert A^+\rVert^2\lVert A\rVert^2\right)\\
&=2 \,f(A)\left(-4\lVert[A^+,A^-]\rVert^2- \lVert A^+\rVert^2\lVert A\rVert^2\right)\,,
\end{aligned}\end{equation}
where the last equality follows by:
$$
\langle [[A^+,A^-],A],A\rangle=-\frac 12\langle [[A,A^*],A],A\rangle=-\frac 12 \langle [A,A^*],[A,A^*]\rangle=-2 \lVert [A^+,A^-]\rVert^2\,.
$$
In particular, if $f(A)$ holds positive along the flow, then the claim follows by standard ODEs arguments. Therefore, let us focus on $f(A_t)$. In the same fashion as \eqref{ev_norm}, we have
$$
\tfrac {d}{dt}f(A)=c\,f(A)(4\lVert[A^+,A^-]\rVert^2+\tfrac 12\lVert A^+\rVert^4)\,, \qquad c=\tfrac{\alpha'}{4}\tfrac{\tau(\tau-1)^2}{8}\,,
$$
and, since $f(A_0)>0$ by hypothesis, we have $f(A_t)>0$ for every $t\in[0,+\infty)$. 

\smallskip
To conclude, we prove the convergence claim. Let us assume $\alpha'=0$. Then, by means of a direct computation, we have
$$
\tfrac {d}{dt}\lVert A^+\rVert^2=-4\lVert[A^+,A^-]\rVert^2-\tfrac 12\lVert A^+\rVert^4\,,
$$
and hence
$$
\tfrac {d}{dt}\lVert A^+\rVert^2\leq-\tfrac 12\lVert A^+\rVert^4\,,
$$
which yields the a priori estimate
$$
\lVert A_t^+\rVert^2 \leq \frac{\lVert A_0^+\rVert^2}{1+\tfrac12 \lVert A_0^+\rVert^2t}\,.
$$ 
Thus, we have $\lVert A_t\rVert^2\to \bar A\in\ \mathfrak u(\n_1)$, as $t\to \infty$, and this concludes the proof. Indeed, convergence in vector space topology yields subconvergence in Cheeger-Gromov topology \cite[Cor. 6.20]{LauLMS}, and $\bar A\in \mathfrak u(\n_1)$ is equivalent to require the metric to be K\"ahler.
\end{proof}

To conclude, we study the Anomaly flow on a concrete example.

\begin{ex}\rm Let $\g=(0,0,0,0,f^{12},f^{13})$ be the unique nilpotent almost-abelian Lie algebra (up to isomorphism) admitting a balanced structure (see \cite{U07}). Let also $J$ be the complex structure on $\g$ such that the $(1,0)$-basis $\{\zeta^1,\zeta^2,\zeta^3\}$ satisfies
$$
d\zeta^1=d\zeta^2=0\,,\qquad d\zeta^3=i\,\zeta^{12}+i\,\zeta^{2\bar 1}\,,
$$
and let $\Psi$ be the closed $(3,0)$-form given by $$\Psi=\zeta^1\wedge\, \zeta^2\wedge\, \zeta^3\,.$$

If we consider a diagonal Hermitian inner product $\nu$ on $(\g,J)$
$$
\nu=  i\, a\, \zeta^{1\bar1}+i\, b\, \zeta^{2\bar2}+i\, c\,\zeta^{3\bar3}\,, \qquad a,b,c>0\,,
$$
then a direct computation yields
$$
\iota_{\nu}( i\partial\overline\partial\, \nu)= i \left(\frac cb\, \zeta^{1\bar 1}+\frac ca\, \zeta^{2\bar 2}-\frac {c^2}{ab}\,\zeta^{3\bar 3}\right)\,,
$$
and $\lVert\Psi\rVert^2_{\nu}=(abc)^{-1}$. Thus, if we consider the initial metric $$\nu_0=  i\, \zeta^{1\bar1}+i\, \zeta^{2\bar2}+i\, \zeta^{3\bar3}\,,$$ the Anomaly flow $\partial_t\,\nu_t = \Vert\Psi\Vert_{\nu_t}^{-2}\, \iota_{\nu_t}( i\partial\overline\partial\, \nu_t)$ reduces to the ODEs system
$$
a'= ac^2\,,\quad b'=bc^2\,,\quad c'= -c^3\,,
$$
whose solution
$$
a_t=b_t=e^{\sqrt{2t+1}-1}\,,\quad c_t=\frac{1}{\sqrt{2t+1}}\,,
$$
is defined for every $t\in(-\tfrac12,+\infty)$.
\end{ex}

\section*{Appendix A}\label{AppA}

In the following we provide the explicit calculations used to compute $\tr(\Omega^\tau\!\wedge \Omega^\tau)$ in Proposition \ref{tr-bas}.

\medskip
Let $\{e^1,\ldots,e^6\}$ be a left-invariant coframe of $G$ satisfying \eqref{on-bas}. Then, by means of \eqref{curvature} and the connection 1-forms $(\sigma^{\tau})^i_j$ obtained in Proposition \ref{tr-bas}, one gets that $d((\sigma^\tau)_i^j)=-d((\sigma^\tau)_j^i)$ and
$$\begin{aligned}
d((\sigma^\tau)_2^1)= &+ \tfrac{(\tau-1)}{4}{\scriptstyle \big(2A_2^2A_2^5+ A_2^3A_3^5+A_2^4A_3^2 \big)}e^{26}+\tfrac{(\tau-1)}{4} {\scriptstyle\big(A_2^2\left(A_2^4+A_3^5\right)-A_2^5 \left(A_2^3+A_3^2\right) \big)}e^{36}\\
&-\tfrac{(\tau-1)}{4} {\scriptstyle\big(A_2^2\left(A_2^3-A_3^2\right)+A_2^5 \left(A_2^4-A_3^5\right) \big)}e^{46}-\tfrac{(\tau-1)}{4}{\scriptstyle \big(2\left(A_2^2\right)^2+\left(A_3^2\right)^2+\left(A_3^5\right)^2+A_2^3A_3^2-A_2^4A_3^5 \big)}e^{56}\,,\\
\\
d((\sigma^\tau)_3^1)= &-\tfrac{(\tau-1)}{4} {\scriptstyle\big(A_2^2\left(A_2^4+A_3^5\right)-A_2^5 \left(A_2^3+A_3^2\right) \big)}e^{26}+ \tfrac{(\tau-1)}{4}{\scriptstyle \big(2A_2^2A_2^5+ A_2^3A_3^5+A_2^4A_3^2 \big)}e^{36}\\
&-\tfrac{(\tau-1)}{4} {\scriptstyle\big(2\left(A_2^2\right)^2+\left(A_2^3\right)^2+\left(A_2^4\right)^2+A_2^3A_3^2-A_2^4A_3^5 \big)}e^{46}-\tfrac{(\tau-1)}{4} {\scriptstyle\big(A_2^2\left(A_2^3-A_3^2\right)+A_2^5 \left(A_2^4-A_3^5\right) \big)}e^{56}\,,\\
\\
d((\sigma^\tau)_4^1)= &+\tfrac{(\tau-1)}{4} {\scriptstyle\big(A_2^2\left(A_2^3-A_3^2\right)+A_2^5 \left(A_2^4-A_3^5\right) \big)}e^{26}+\tfrac{(\tau-1)}{4} {\scriptstyle\big(2\left(A_2^2\right)^2+\left(A_2^3\right)^2+\left(A_2^4\right)^2+A_2^3A_3^2-A_2^4A_3^5 \big)}e^{36}
\\&+ \tfrac{(\tau-1)}{4}{\scriptstyle \big(2A_2^2A_2^5+ A_2^3A_3^5+A_2^4A_3^2 \big)}e^{46}-\tfrac{(\tau-1)}{4} {\scriptstyle\big(A_2^2\left(A_2^4+A_3^5\right)-A_2^5 \left(A_2^3+A_3^2\right) \big)}e^{56}\,,\\
\\
d((\sigma^\tau)_5^1)= &+\tfrac{(\tau-1)}{4}{\scriptstyle \big(2\left(A_2^2\right)^2+\left(A_3^2\right)^2+\left(A_3^5\right)^2+A_2^3A_3^2-A_2^4A_3^5 \big)}e^{26}+\tfrac{(\tau-1)}{4} {\scriptstyle\big(A_2^2\left(A_2^3-A_3^2\right)+A_2^5 \left(A_2^4-A_3^5\right) \big)}e^{36}
\\&+\tfrac{(\tau-1)}{4} {\scriptstyle\big(A_2^2\left(A_2^4+A_3^5\right)-A_2^5 \left(A_2^3+A_3^2\right) \big)}e^{46}+ \tfrac{(\tau-1)}{4}{\scriptstyle \big(2A_2^2A_2^5+ A_2^3A_3^5+A_2^4A_3^2 \big)}e^{56}\,,
\end{aligned}$$
together with the following relations:
$$\begin{aligned}
d((\sigma^\tau)_2^1)=-d((\sigma^\tau)_6^5)\,,\quad d((\sigma^\tau)_3^1)=-d((\sigma^\tau)_6^4)\,,\quad d((\sigma^\tau)_4^1)=d((\sigma^\tau)_6^3)\,,\quad d((\sigma^\tau)_5^1)=d((\sigma^\tau)_6^2)\,.
\end{aligned}$$

On the other hand, if we denote by $(\Lambda^\tau)^i_j:=\sum(\sigma^\tau)^i_k\wedge (\sigma^\tau)^k_j$, it follows $(\Lambda^\tau)^i_j=-(\Lambda^\tau)^j_i$ and

$$\begin{aligned}
(\Lambda^\tau)_2^1= & -\tfrac{\tau(\tau-1)}{8}{\scriptstyle \left(4\left(A_2^2\right)^2+\left(A_2^3+A_3^2\right)^2+\left(A_2^4-A_3^5\right)^2 \right)}e^{12}- \tfrac{(\tau-1)}{4}{\scriptstyle \big(2A_2^2A_2^5+ A_2^3A_3^5+A_2^4A_3^2 \big)}e^{26}\\
&+\tfrac{(\tau-1)}{4} {\scriptstyle\big(A_2^2\left(A_2^4+A_3^5\right)-A_2^5 \left(A_2^3+A_3^2\right) \big)}e^{36}-\tfrac{(\tau-1)}{4} {\scriptstyle\big(A_2^2\left(A_2^3-A_3^2\right)+A_2^5 \left(A_2^4-A_3^5\right) \big)}e^{46}\\
&+\tfrac{(\tau-1)}{8} {\scriptstyle\big(\left(A_2^3\right)^2-\left(A_3^2\right)^2+\left(A_2^4\right)^2-\left(A_3^5\right)^2 \big)}e^{56}\,,\\
\\
(\Lambda^\tau)_3^1= &-\tfrac{\tau(\tau-1)}{8}{\scriptstyle \left(4\left(A_2^2\right)^2+\left(A_2^3+A_3^2\right)^2+\left(A_2^4-A_3^5\right)^2 \right)}e^{13}-\tfrac{(\tau-1)}{4} {\scriptstyle\big(A_2^2\left(A_2^4+A_3^5\right)-A_2^5 \left(A_2^3+A_3^2\right) \big)}e^{26}\\
&- \tfrac{(\tau-1)}{4}{\scriptstyle \big(2A_2^2A_2^5+ A_2^3A_3^5+A_2^4A_3^2 \big)}e^{36}-\tfrac{(\tau-1)}{8} {\scriptstyle\big(\left(A_2^3\right)^2-\left(A_3^2\right)^2+\left(A_2^4\right)^2-\left(A_3^5\right)^2 \big)}e^{46}\\
&+\tfrac{(\tau-1)}{4} {\scriptstyle\big(A_2^2\left(A_2^3-A_3^2\right)+A_2^5 \left(A_2^4-A_3^5\right) \big)}e^{56}\,,\\
\\
(\Lambda^\tau)_3^2= &+\tau{\scriptstyle\big(A_2^2\left(A_2^4+A_3^5\right)-A_2^5 \left(A_2^3+A_3^2\right) \big)}e^{16}+\tfrac{(\tau-1)^2}{16}{\scriptstyle \left(4\left(A_2^2\right)^2+\left(A_2^3+A_3^2\right)^2-\left(A_2^4+A_3^5\right)^2 \right)}\left(e^{23}-e^{45}\right)\\
&+\tfrac{(\tau-1)^2}{8}{\scriptstyle \left(\left(A_2^3+A_3^2\right)\left(A_2^4-A_3^5\right) \right)}\left(e^{24}+e^{35}\right)+\tfrac{(\tau-1)^2}{4}{\scriptstyle \left(A_2^2\left(A_2^4-A_3^5\right) \right)}\left(e^{25}-e^{34}\right)\,,\\
\\
(\Lambda^\tau)_4^1= &-\tfrac{\tau(\tau-1)}{8}{\scriptstyle \left(4\left(A_2^2\right)^2+\left(A_2^3+A_3^2\right)^2+\left(A_2^4-A_3^5\right)^2 \right)}e^{14}+\tfrac{(\tau-1)}{4} {\scriptstyle\big(A_2^2\left(A_2^3-A_3^2\right)+A_2^5 \left(A_2^4-A_3^5\right) \big)}e^{26}\\
&+\tfrac{(\tau-1)}{8} {\scriptstyle\big(\left(A_2^3\right)^2-\left(A_3^2\right)^2+\left(A_2^4\right)^2-\left(A_3^5\right)^2 \big)}e^{36}- \tfrac{(\tau-1)}{4}{\scriptstyle \big(2A_2^2A_2^5+ A_2^3A_3^5+A_2^4A_3^2 \big)}e^{46}\\
&-\tfrac{(\tau-1)}{4} {\scriptstyle\big(A_2^2\left(A_2^4+A_3^5\right)-A_2^5 \left(A_2^3+A_3^2\right) \big)}e^{56}\,,\\
\\
(\Lambda^\tau)_4^2= &-\tau{\scriptstyle\big(A_2^2\left(A_2^3-A_3^2\right)+A_2^5 \left(A_2^4-A_3^5\right) \big)}e^{16}+\tfrac{(\tau-1)^2}{8}{\scriptstyle \left(\left(A_2^3+A_3^2\right)\left(A_2^4-A_3^5\right) \right)}\left(e^{23}-e^{45}\right)\\
&+\tfrac{(\tau-1)^2}{16}{\scriptstyle \left(4\left(A_2^2\right)^2+\left(A_2^3+A_3^2\right)^2-\left(A_2^4+A_3^5\right)^2 \right)}\left(e^{24}+e^{35} \right)-\tfrac{(\tau-1)^2}{4}{\scriptstyle \left(A_2^2\left(A_2^3+A_3^2\right) \right)}\left(e^{25}-e^{34}\right)\,,
\end{aligned}$$
$$\begin{aligned}
(\Lambda^\tau)_4^3= &- \tfrac{\tau}{2} {\scriptstyle\big(\left(A_2^3\right)^2-\left(A_3^2\right)^2+\left(A_2^4\right)^2-\left(A_3^5\right)^2 \big)}e^{16}-\tfrac{(\tau-1)^2}{4}{\scriptstyle \left(A_2^2\left(A_2^4-A_3^5\right) \right)}\left(e^{23}-e^{45}\right)\\
&+\tfrac{(\tau-1)^2}{4}{\scriptstyle \left(A_2^2\left(A_2^3+A_3^2\right) \right)}\left(e^{24}+e^{35}\right)-\tfrac{(\tau-1)^2}{8}{\scriptstyle \left(\left(A_2^3+A_3^2\right)^2+\left(A_2^4-A_3^5\right)^2 \right)}e^{25}-\tfrac{(\tau-1)^2}{2}{\scriptstyle\left(A_2^2\right)^2}e^{34}\,,\\
\\
(\Lambda^\tau)_5^1= &-\tfrac{\tau(\tau-1)}{8}{\scriptstyle \left(4\left(A_2^2\right)^2+\left(A_2^3+A_3^2\right)^2+\left(A_2^4-A_3^5\right)^2 \right)}e^{15}-\tfrac{(\tau-1)}{8} {\scriptstyle\big(\left(A_2^3\right)^2-\left(A_3^2\right)^2+\left(A_2^4\right)^2-\left(A_3^5\right)^2 \big)}e^{26}\\
&+\tfrac{(\tau-1)}{4} {\scriptstyle\big(A_2^2\left(A_2^3-A_3^2\right)+A_2^5 \left(A_2^4-A_3^5\right) \big)}e^{36}+\tfrac{(\tau-1)}{4} {\scriptstyle\big(A_2^2\left(A_2^4+A_3^5\right)-A_2^5 \left(A_2^3+A_3^2\right) \big)}e^{46}\\
&- \tfrac{(\tau-1)}{4}{\scriptstyle \big(2A_2^2A_2^5+ A_2^3A_3^5+A_2^4A_3^2 \big)}e^{56}\,,\\
\\
(\Lambda^\tau)_5^2= &+ \tfrac{\tau}{2} {\scriptstyle\big(\left(A_2^3\right)^2-\left(A_3^2\right)^2+\left(A_2^4\right)^2-\left(A_3^5\right)^2 \big)}e^{16}+\tfrac{(\tau-1)^2}{4}{\scriptstyle \left(A_2^2\left(A_2^4-A_3^5\right) \right)}\left(e^{23}-e^{45}\right)\\
&-\tfrac{(\tau-1)^2}{4}{\scriptstyle \left(A_2^2\left(A_2^3+A_3^2\right) \right)}\left(e^{24}+e^{35}\right)-\tfrac{(\tau-1)^2}{2}{\scriptstyle\left(A_2^2\right)^2}e^{25}-\tfrac{(\tau-1)^2}{8}{\scriptstyle \left(\left(A_2^3+A_3^2\right)^2+\left(A_2^4-A_3^5\right)^2 \right)}e^{34}\,,\\
\\
(\Lambda^\tau)_6^1= &+\tfrac{(\tau-1)^2}{8}{\scriptstyle \left(4\left(A_2^2\right)^2+\left(A_2^3+A_3^2\right)^2+\left(A_2^4-A_3^5\right)^2 \right)}\left(e^{25}+e^{34}\right)\,,
\end{aligned}$$
together with the following relations:
$$\begin{aligned}
&(\Lambda^\tau)_2^1=-(\Lambda^\tau)_6^5\,,\quad (\Lambda^\tau)_3^1=-(\Lambda^\tau)_6^4\,,\quad (\Lambda^\tau)_4^1=(\Lambda^\tau)_6^3\,,\quad (\Lambda^\tau)_5^1=(\Lambda^\tau)_6^2\,,\\
&(\Lambda^\tau)_3^2=-(\Lambda^\tau)_5^4\,,\quad (\Lambda^\tau)_4^2=-(\Lambda^\tau)_5^3\,.
\end{aligned}$$

\section*{Appendix B}\label{AppB}

In this Appendix we provide the curvature 2-forms $(\Omega^\tau)_i^j$ of a Gauduchon connection $\nabla^\tau$ on an almost-abelian Lie group $(G,\omega,J)$.

\medskip
Let $\{e^1,\ldots,e^6\}$ be a left-invariant coframe of $G$ satisfying \eqref{on-bas}. Then, by means of \eqref{curvature} and the connection 1-forms $(\sigma^{\tau})^i_j$ obtained in Proposition \ref{tr-bas}, one gets that $(\Omega^\tau)_i^j=-(\Omega^\tau)_j^i$ and
$$\begin{aligned}
(\Omega^\tau)_2^1= & -\tfrac{\tau(\tau-1)}{8}{\scriptstyle \left(4\left(A_2^2\right)^2+\left(A_2^3+A_3^2\right)^2+\left(A_2^4-A_3^5\right)^2 \right)}e^{12}+\tfrac{(\tau-1)}{2} {\scriptstyle\big(A_2^2\left(A_2^4+A_3^5\right)-A_2^5 \left(A_2^3+A_3^2\right) \big)}e^{36}\\
&-\tfrac{(\tau-1)}{2} {\scriptstyle\big(A_2^2\left(A_2^3-A_3^2\right)+A_2^5 \left(A_2^4-A_3^5\right) \big)}e^{46}-\tfrac{(\tau-1)}{8}{\scriptstyle \big(4\left(\left(A_2^2\right)^2+\left(A_3^2\right)^2+\left(A_3^5\right)^2\right)-\left(A_2^3-A_3^2\right)^2- \left(A_2^4+A_3^5\right)^2 \big)}e^{56}\,,\\
\\
(\Omega^\tau)_3^1= &-\tfrac{\tau(\tau-1)}{8} {\scriptstyle\big(4\left(A_2^2\right)^2+\left(A_2^3+A_3^2\right)^2+\left(A_2^4-A_3^5\right)^2 \big)}e^{13}-\tfrac{(\tau-1)}{2} {\scriptstyle\big(A_2^2\left(A_2^4+A_3^5\right)-A_2^5 \left(A_2^3+A_3^2\right) \big)}e^{26}\\
&-\tfrac{(\tau-1)}{8} {\scriptstyle\big(4\left(\left(A_2^2\right)^2+\left(A_2^3\right)^2+\left(A_2^4\right)^2\right)-\left(A_2^3-A_3^2\right)^2- \left(A_2^4+A_3^5\right)^2 \big)}e^{46}-\tfrac{(\tau-1)}{2} {\scriptstyle\big(A_2^2\left(A_2^3-A_3^2\right)+A_2^5 \left(A_2^4-A_3^5\right) \big)}e^{56}\,,\\
\\
(\Omega^\tau)_3^2=&+{\scriptstyle\tau  \big(A_2^2\left(A_2^4+A_3^5\right)-A_2^5 \left(A_2^3+A_3^2\right) \big)}e^{16}+\tfrac{(\tau-1)^2}{16}{\scriptstyle \big(4\left(A_2^2\right)^2+\left(A_2^3+A_3^2\right)^2-\left(A_2^4-A_3^5\right)^2 \big)}\left(e^{23}-e^{45}\right)\\
&+\tfrac{(\tau-1)^2}{8} {\scriptstyle\big(\left(A_2^3+A_3^2\right)\left(A_2^4-A_3^5\right) \big)}\left(e^{24}+e^{35}\right)+\tfrac{(\tau-1)^2}{4} {\scriptstyle\big(A_2^2\left(A_2^4-A_3^5\right) \big)}\left(e^{25}-e^{34}\right)\,,\\
\\
(\Omega^\tau)_4^1= &-\tfrac{\tau(\tau-1)}{8}{\scriptstyle \big(4\left(A_2^2\right)^2+\left(A_2^3+A_3^2\right)^2+\left(A_2^4-A_3^5\right)^2 \big)}e^{14}-\tfrac{(\tau-1)}{2} {\scriptstyle\big(A_2^2\left(A_2^4+A_3^5\right)-A_2^5 \left(A_2^3+A_3^2\right) \big)}e^{56}\\
&+\tfrac{(\tau-1)}{2}{\scriptstyle \big(A_2^2\left(A_2^3-A_3^2\right)+A_2^5 \left(A_2^4-A_3^5\right) \big)}e^{26}+\tfrac{(\tau-1)}{8}{\scriptstyle \big(4\left(\left(A_2^2\right)^2+\left(A_2^3\right)^2+\left(A_2^4\right)^2\right)-\left(A_2^3-A_3^2\right)^2- \left(A_2^4+A_3^5\right)^2 \big)}e^{36}\,,\\
\\
(\Omega^\tau)_4^2=&-{\scriptstyle\tau  \big(A_2^2\left(A_2^3-A_3^2\right)+A_2^5 \left(A_2^4-A_3^5\right) \big)}e^{16}+\tfrac{(\tau-1)^2}{16} {\scriptstyle\big(4\left(A_2^2\right)^2-\left(A_2^3+A_3^2\right)^2+\left(A_2^4-A_3^5\right)^2 \big)}\left(e^{24}+e^{35}\right)\\
&+\tfrac{(\tau-1)^2}{8} {\scriptstyle\big(\left(A_2^3+A_3^2\right)\left(A_2^4-A_3^5\right) \big)\left(e^{23}-e^{45}\right)}-\tfrac{(\tau-1)^2}{4} {\scriptstyle\big(A_2^2\left(A_2^3+A_3^2\right) \big)}\left(e^{25}-e^{34}\right)\,,\\
\\
(\Omega^\tau)_4^3=&-\tfrac{\tau}{2} {\scriptstyle\big(\left(A_2^3\right)^2-\left(A_3^2\right)^2+\left(A_2^4\right)^2-\left(A_3^5\right)^2 \big)}e^{16}-\tfrac{(\tau-1)^2}{4} {\scriptstyle\big(A_2^2\left(A_2^4-A_3^5\right) \big)}\left(e^{23}-e^{45}\right)\\
&+\tfrac{(\tau-1)^2}{4} {\scriptstyle\big(A_2^2\left(A_2^3+A_3^2\right) \big)}\left(e^{24}+e^{35}\right)-\tfrac{(\tau-1)^2}{8} {\scriptstyle\big(\left(A_2^3+A_3^2\right)^2+\left(A_2^4-A_3^5\right)^2 \big)}e^{25}-\tfrac{(\tau-1)^2}{2}{\scriptstyle\left(A_2^2\right)^2}e^{34}\,,
\end{aligned}$$
$$\begin{aligned}
(\Omega^\tau)_5^1= &-\tfrac{\tau(\tau-1)}{8}{\scriptstyle \big(4\left(A_2^2\right)^2+\left(A_2^3+A_3^2\right)^2+\left(A_2^4-A_3^5\right)^2 \big)}e^{15}+\tfrac{(\tau-1)}{2} {\scriptstyle\big(A_2^2\left(A_2^4+A_3^5\right)-A_2^5 \left(A_2^3+A_3^2\right) \big)}e^{46}\\
&+\tfrac{(\tau-1)}{2} {\scriptstyle\big(A_2^2\left(A_2^3-A_3^2\right)+A_2^5 \left(A_2^4-A_3^5\right) \big)}e^{36}+\tfrac{(\tau-1)}{8} {\scriptstyle\big(4\left(\left(A_2^2\right)^2+\left(A_3^2\right)^2+\left(A_3^5\right)^2\right)-\left(A_2^3-A_3^2\right)^2- \left(A_2^4+A_3^5\right)^2 \big)}e^{26}\,,\\
\\
(\Omega^\tau)_5^2=&+\tfrac{\tau}{2} {\scriptstyle\big(\left(A_2^3\right)^2-\left(A_3^2\right)^2+\left(A_2^4\right)^2-\left(A_3^5\right)^2 \big)}e^{16}+\tfrac{(\tau-1)^2}{4} {\scriptstyle\big(A_2^2\left(A_2^4-A_3^5\right) \big)}\left(e^{23}-e^{45}\right)\\
&-\tfrac{(\tau-1)^2}{4} {\scriptstyle\big(A_2^2\left(A_2^3+A_3^2\right) \big)}\left(e^{24}+e^{35}\right)-\tfrac{(\tau-1)^2}{8} {\scriptstyle\big(\left(A_2^3+A_3^2\right)^2+\left(A_2^4-A_3^5\right)^2 \big)}e^{34}-\tfrac{(\tau-1)^2}{2}{\scriptstyle\left(A_2^2\right)^2}e^{25}\,,\\
\\
(\Omega^\tau)_6^1=&+\tfrac{(\tau-1)^2}{8} {\scriptstyle\big(4\left(A_2^2\right)^2+\left(A_2^3+A_3^2\right)^2+\left(A_2^4-A_3^5\right)^2 \big)}\left(e^{25}+e^{34}\right)\,,
\end{aligned}$$
together with the following relations:
$$\begin{aligned}
&(\Omega^\tau)_6^5=-(\Omega^\tau)_2^1\,,\quad (\Omega^\tau)_6^4=-(\Omega^\tau)_3^1\,,\quad (\Omega^\tau)_6^3=(\Omega^\tau)_4^1\,,\quad (\Omega^\tau)_6^2=(\Omega^\tau)_5^1\,,\\
&(\Omega^\tau)_5^4=-(\Omega^\tau)_3^2\,,\quad (\Omega^\tau)_5^3=(\Omega^\tau)_4^2\,.
\end{aligned}$$

\bibliographystyle{abbrv}

\end{document}